\documentclass[12pt]{article}
\usepackage{t1enc}
\usepackage[latin1]{inputenc}
\usepackage[english]{babel}
\usepackage{amsmath,amsthm}
\usepackage{amsfonts}
\usepackage{latexsym}
\usepackage[dvips]{graphicx}
\usepackage{graphicx}
\usepackage[natural]{xcolor}
\newtheorem{thm}{Theorem}[section]
\newtheorem{cor}[thm]{Corollary}

\newtheorem{lem}[thm]{Lemma}

\newtheorem{conj}[thm]{Conjecture}

\numberwithin{equation}{section}

\newcommand{\mad}{\mathrm{mad}}

\usepackage[varg]{txfonts}
\usepackage{graphicx, epsfig, subfigure}
\usepackage{enumerate} 
\usepackage[square, numbers, sort&compress]{natbib} 

\numberwithin{equation}{section}

\begin{document}
\title{Equitable vertex arboricity of graphs\thanks{This research is supported by National Natural Science Foundation of China (No. 10971121, 11201440).}}
\author{Jian-Liang Wu\unskip\textsuperscript{a}\thanks{Email address: jlwu@sdu.edu.cn}~, ~Xin Zhang\unskip\textsuperscript{a,b},  ~Hailuan Li\unskip\textsuperscript{a}\\[.5em]
{\small \textsuperscript{a}\unskip School of Mathematics, Shandong University, Jinan 250100, P. R. China}\\
{\small \textsuperscript{b}\unskip Department of Mathematics, Xidian University, Xi'an 710071, P. R. China}\\
}
\date{}
\maketitle

\begin{abstract}
\baselineskip  0.6cm
An equitable $(t,k,d)$-tree-coloring of a graph $G$ is a coloring to vertices of $G$ such that the sizes of any two color classes differ by at most one and the subgraph induced by each color class is a forest of maximum degree at most $k$ and diameter at most $d$. The minimum $t$ such that $G$ has an equitable $(t',k,d)$-tree-coloring for every $t'\geq t$ is called the strong equitable $(k,d)$-vertex-arboricity and denoted by $va^{\equiv}_{k,d}(G)$.
In this paper, we give sharp upper bounds for $va^{\equiv}_{1,1}(K_{n,n})$ and $va^{\equiv}_{k,\infty}(K_{n,n})$ by showing that $va^{\equiv}_{1,1}(K_{n,n})=O(n)$ and $va^{\equiv}_{k,\infty}(K_{n,n})=O(n^{\frac{1}{2}})$ for every $k\geq 2$.
It is also proved that $va^{\equiv}_{\infty,\infty}(G)\leq 3$ for every planar graph $G$ with girth at least 5 and $va^{\equiv}_{\infty,\infty}(G)\leq 2$ for every planar graph $G$ with girth at least 6 and for every outerplanar graph. We conjecture that $va^{\equiv}_{\infty,\infty}(G)=O(1)$ for every planar graph and  $va^{\equiv}_{\infty,\infty}(G)\leq \lceil\frac{\Delta(G)+1}{2}\rceil$ for every graph $G$.
\\[.5em]

\end{abstract}

\baselineskip 0.6cm
\section{Introduction}
All graphs considered in the paper are finite, simple and undirected. We use $V(G)$, $E(G)$, $\delta(G)$ and $\Delta(G)$ to denote the set of vertices, the set of edges, the minimum degree and the maximum degree of $G$, respectively. $N_G(v)$ denotes the set of neighbors of a vertex $v$ in $G$ and $d_G(v)=|N_G(v)|$ denotes the degree of $v$. Sometimes we use $d(v)$ instead of $d_G(v)$ for brevity. A $k$-, $k^+$- and $k^-$-$vertex$ in $G$ is a vertex of degree $k$, at least $k$ and at most $k$, respectively. If $uv\in E(G)$ and $d(u)=k$, then we say that $u$ is a $k$-$neighbor$ of $v$; $k^-$-$neighbor$ and $k^+$-$neighbor$ can be similarly defined. For other undefined concepts we refer the reader to \cite{Bondy.2008}.

We associate positive integers $1,2,\cdots,t$ with colors, and call $f$ a $t$-$coloring$ of $G$ if $f$ is a mapping from $V(G)$ to $\{1,2,\cdots,t\}$. For $1\leq i\leq t$, let $V_i=\{v~|~f(v)=i\}$. A $t$-coloring $f$ of $G$ is $equitable$ if $||V_i|-|V_j||\leq 1$ for all $i$ and $j$, that is, every color class has size $\lfloor\frac{|V(G)|}{t}\rfloor$ or $\lceil\frac{|V(G)|}{t}\rceil$. A $t$-coloring of $G$ is $proper$ if every two adjacent vertices have the different colors. The smallest number $t$ such that $G$ has a proper equitable $t$-coloring, denoted by $\chi^{=}(G)$, is the $equitable$ $chromatic$ $number$. Note that a proper equitable $t$-colorable graph may admit no proper equitable $t'$-colorings for some $t'>t$. For example, the complete bipartite graph $H:=K_{2m+1,2m+1}$ has no proper equitable $(2m+1)$-colorings, although it satisfies $\chi^{=}(H)=2$. This fact motivates us to introduce another interesting parameter for proper equitable coloring. The $equitable$ $chromatic$ $threshold$ of $G$, denoted by $\chi^{\equiv}(G)$, is the smallest integer $t$ such that $G$ has proper equitable colorings for any number of colors greater than or equal to $t$. In 1970, Hajnal and Szemerédi \cite{Hajnal} answered a question of Erd\H{o}s by proving that every graph $G$ with $\Delta(G)\leq r$ has a proper equitable $(r+1)$-coloring. In fact, Hajnal-Szemerrédi Theorem implies $\chi^{\equiv}(G)\leq \Delta(G)+1$ for every graph $G$. In 2008, Kierstead and Kostochka \cite{Kierstead} simplified the proof of
Hajnal-Szemerrédi Theorem, and moreover, they \cite{KiersteadJCT} strengthened Hajnal-Szemerrédi Theorem by proving that $G$ has a proper equitable $(r+1)$-coloring if $G$ is a graph such that $d(x) + d(y)\leq 2r + 1$ for every edge $xy$.

Regarding equitable colorings, there are two well-known conjectures. Note that Conjecture \ref{CLW} is stronger than Conjecture \ref{ECC}.

\begin{conj}{\rm \cite{Meye}} \label{ECC}
For any connected graph $G$, except the complete graph and the odd cycle, $\chi^{=}(G)\leq \Delta(G)$.
\end{conj}
\begin{conj}\label{CLW} {\rm \cite{C-L-W}}
For any connected graph $G$, except the complete graph, the odd cycle and the complete bipartite graph $K_{2m+1,2m+1}$, $\chi^{\equiv}(G)\leq \Delta(G)$.
\end{conj}
\noindent


The above two conjectures have been confirmed for many classes of graphs, such as graphs with $\Delta\leq 3$ \cite{C-L-W,Chen} or $\Delta\geq \frac{|V(G)|}{3}+1$ \cite{C-L-W,Chen,Yap}, bipartite graphs \cite{Lih.bip}, outerplanar graphs \cite{Yap}, series-parallel graphs \cite{ZW} and planar graphs with $\Delta\geq 9$ \cite{K,Yap-planar}. There are other related results, see \cite{Wang,Zhu}.


In \cite{Fan}, Fan, Kierstead, Liu, Molla, Wu and Zhang first considered relaxed equitable coloring of graphs. They proved that every graph has an equitable $\Delta$-coloring such that each color class induces a forest with maximum degree at most one. On the basis of this research, we aim to introduce the notion of equitable $(t,k,d)$-tree-coloring. A $t$-coloring $f$ of a graph $G$ is a $(t,k,d)$-$tree$-$coloring$ of $G$ if each component of $G[V_i]$ is a tree of maximum degree at most $k$ and diameter at most $d$. Sometimes, a $(t,\infty,\infty)$-tree-coloring is called a $t$-tree-coloring for short. The $(k,d)$-$vertex$ $arboricity$ of $G$, denoted by $va_{k,d}(G)$, is the minimum $t$ such that $G$ has a $(t,k,d)$-tree-coloring. Indeed, the notion of $(t,k,d)$-tree-coloring is a uniform form of some familiar kinds of vertex coloring. For example, it is obvious that $va_{0,0}(G)=\chi(G)$, $va_{2,\infty}(G)=vla(G)$ and $va_{\infty,\infty}(G)=va(G)$, where $\chi(G)$ is the standard chromatic number, $vla(G)$ is the vertex linear arboricity and $va(G)$ is the vertex arboricity of $G$. It is also trivial that $va_{k,d}(K_{m,n})=2$ for complete bipartite graph $K_{m,n}$ and integers $k,d\geq 0$. In \cite{Chartrand}, it was prove that the set of vertices of every planar graph can be partitioned into three subsets such that each subset induces a forest. This implies $va_{\infty,\infty}(G)\leq 3$ for every planar graph $G$.

An $equitable$ $(t,k,d)$-$coloring$ is a $(t,k,d)$-coloring that is equitable. The $equitable$ $(k,d)$-$vertex$ $arboricity$ of a graph $G$, denoted by $va^{=}_{k,d}(G)$, is the smallest $t$ such that $G$ has an equitable $(t,k,d)$-tree-coloring. The $strong$ $equitable$ $(k,d)$-$vertex$ $arboricity$ of $G$, denoted by $va^{\equiv}_{k,d}(G)$, is the smallest $t$ such that $G$ has an equitable $(t',k,d)$-coloring for every $t'\geq t$. It is clear that $va^{=}_{0,0}(G)=\chi^{=}(G)$ and $va^{\equiv}_{0,0}(G)=\chi^{\equiv}(G)$ for every graph $G$. In view of this, for a graph $G$, $va^{=}_{k,d}(G)$ and $va^{\equiv}_{k,d}(G)$ may vary a lot.


In Section 2, we investigate the strong equitable $(k,d)$-vertex arboricity of the complete bipartite graph $K_{n,n}$ by showing that $va^{\equiv}_{1,1}(K_{n,n})=O(n)$ and $va^{\equiv}_{k,\infty}(K_{n,n})=O(n^{\frac{1}{2}})$ for every $k\geq 2$. In Section 3, we consider planar graphs and prove that $va^{\equiv}_{\infty,\infty}(G)=O(1)$ for every planar graph with girth at least 5.


\section{Complete bipartite graphs}


\begin{lem}\label{lem:even}
The complete bipartite graph $K_{n,n}$ has an equitable $(t,k,d)$-tree-coloring for every even integer $t\geq 2$.
\end{lem}

\begin{proof}
One can easily construct an equitable $(t, k, d)$-tree-coloring of $K_{n,n}$ by dividing each partite set into $t/2$ classes equitably and coloring the vertices of each class with one color.
\end{proof}

\begin{thm}\label{thm:main1}
If $K_{n,n}$ is a complete bipartite graph and $k\geq 2$, then  $va^{\equiv}_{1,1}(K_{n,n})\leq 2\lfloor\frac{n+1}{3}\rfloor$, and furthermore, this bound is sharp.
\end{thm}

\begin{proof}
 By Lemma \ref{lem:even}, in order to show $va^{\equiv}_{1,1}(K_{n,n})\leq 2\lfloor\frac{n+1}{3}\rfloor$, we only need to prove that $K_{n,n}$ has an equitable $(q,1,1)$-tree-coloring for every odd $q\geq 2\lfloor\frac{n+1}{3}\rfloor+1$. Note that $3q-2n\geq 6\lfloor\frac{n+1}{3}\rfloor+3-2n\geq 6\times\frac{n-1}{3}+3-2n\geq 1$. Let $X$ and $Y$ be the partite sets of $K_{n,n}$ and let $e=xy$ be an edge of $K_{n,n}$ with $x\in X$ and $y\in Y$. If $q\geq n$, then color $x$ and $y$ with 1, divide each of $X\backslash\{x\}$ and $Y\backslash\{y\}$ into $\frac{q-1}{2}$ classes equitably and color the vertices of each class with a color in $\{2,\cdots,q\}$. One can easily check that the resulting coloring is an equitable $(q,1,1)$-tree-coloring of $K_{n,n}$ with the size of each color class being at most 2. Thus, we assume $q<n$. Suppose $2n=aq+r$, where $0\leq r\leq a-1$. Since $a=\frac{2n-r}{q}\leq \frac{2n}{q}\leq \frac{2n}{2\lfloor\frac{n+1}{3}\rfloor+1}<3$, $a\leq 2$. Now arbitrarily choose $3q-2n$ vertex-disjoint edges from $K_{n,n}$ and color the two end-vertices of each edge with a color in $\{1,\cdots,3q-2n\}$. Let $X'$ and $Y'$ be the uncolored vertices in $X$ and $Y$, respectively. One can see that $|X'|=|Y'|=n-(3q-2n)=3(n-q)>0$. Thus, we can divide each of $X'$ and $Y'$ into $n-q$ classes equitably and color the vertices of each class with a color in $\{3q-2n+1,\cdots,q\}$. It is also easy to check that the resulting coloring of $K_{n,n}$ is an equitable $(q,1,1)$-tree-coloring with the size of each color class being either 2 or 3. Hence $va^{\equiv}_{1,1}(K_{n,n})\leq 2\lfloor\frac{n+1}{3}\rfloor$.
To show this bound is sharp, we investigate the graph $G:=K_{n,n}$ with $n=3t+2$. If $G$ has an equitable $(2t+1,1,1)$-tree-coloring $c$, then the size of every color class in $c$ is at least 3 because $\lceil\frac{2n}{2t+1}\rceil=\lceil\frac{6t+4}{2t+1}\rceil\geq 4$. This implies that there is no edge in $G$ with its two end-vertices colored with the same color. Thus the vertices of every color class forms an independent set. Without loss of generality, suppose there are at least $t+1$ colors appearing in $X$. We then have $|X|\geq 3(t+1)=(3t+2)+1=|X|+1$, a contradiction. This implies $va^{\equiv}_{1,1}(G)\geq 2t+2=2\lfloor\frac{n+1}{3}\rfloor$ and thus $va^{\equiv}_{1,1}(G)=2\lfloor\frac{n+1}{3}\rfloor$.
\end{proof}

In the following we investigate the strong equitable $(\infty, k)$-vertex arboricity of $K_{n,n}$, where $k\geq 2$.
One can see that the diameter of every induced forest in $K_{n,n}$ is at most 2, so an equitable $(\infty, k)$-tree-coloring of $K_{n,n}$ is equivalent to an equitable $(\infty, 2)$-tree-coloring of $K_{n,n}$, that is, $va^{\equiv}_{\infty,k }(K_{n,n})=va^{\equiv}_{\infty, 2}(K_{n,n})$.


Let $K_{n,n}$ be a complete bipartite graph with two partite sets $X$ and $Y$. For a partial $q$-coloring $c$ (not needed to be proper) of $K_{n,n}$, let $V_1,\cdots,V_q$ be its color classes, $a=\lfloor\frac{2n}{q}\rfloor$ and let
\begin{align*}
c(X_1)& =\{V_i~|~|V_i\cap X|=a+1, |V_i\cap Y|=0\},
c(X_2) =\{V_i~|~|V_i\cap X|=a, |V_i\cap Y|=0\},\\
c(X'_1)& =\{V_i~|~|V_i\cap X|=a, |V_i\cap Y|=1\},
c(X'_2) =\{V_i~|~|V_i\cap X|=a-1, |V_i\cap Y|=1\},\\
c(Y_1)& =\{V_i~|~|V_i\cap Y|=a+1, |V_i\cap X|=0\},
c(Y_2) =\{V_i~|~|V_i\cap Y|=a, |V_i\cap X|=0\},\\
c(Y'_1)& =\{V_i~|~|V_i\cap Y|=a, |V_i\cap X|=1\},
c(Y'_2) =\{V_i~|~|V_i\cap Y|=a-1, |V_i\cap X|=1\}.
\end{align*}
We have the following lemma.

\begin{lem}\label{lem:eq.iff}
If $K_{n,n}$ is a complete bipartite graph with partite sets $X$ and $Y$, where $2n=aq+r$ and $0\leq r\leq a-1$, and $c$ is a  partial $q$-coloring of $K_{n,n}$, then
$c$ is an equitable $(q,\infty, 2)$-tree-coloring of $K_{n,n}$ if and only if
\begin{align}
& (a+1)|c(X_1)|+a|c(X_2)|+a|c(X'_1)|+(a-1)|c(X'_2)|+|c(Y'_1)|+|c(Y'_2)|=n,\label{eq:2}\\
& (a+1)|c(Y_1)|+a|c(Y_2)|+a|c(Y'_1)|+(a-1)|c(Y'_2)|+|c(X'_1)|+|c(X'_2)|=n.\label{eq:3}
\end{align}
\end{lem}

\begin{proof}
Let $V_1,\cdots,V_q$ be the color classes of $c$. First suppose that $c$ is an equitable $(q,\infty,2)$-tree-coloring of $K_{n,n}$. Since $2n=aq+r$, the size of each color class of $c$ is either $a$ or $a+1$.
It is easy to see that $\min\{|V_i\cap X|,|V_i\cap Y|\}\leq 1$ for every $1\leq i\leq q$, because otherwise we would find a 4-cycle in some color class $V_i$, a contradiction. Thus
\begin{align}
c(X_1)\cup c(X_2)\cup c(X'_1)\cup c(X'_2)\cup c(Y_1)\cup c(Y_2)\cup c(Y'_1)\cup c(Y'_2)=\bigcup_{i=1}^q V_i \label{cup}
\end{align}
and the equations (\ref{eq:2}) and (\ref{eq:3}) hold accordingly. On the other hand, if equations (\ref{eq:2}) and (\ref{eq:3}) hold, then $c$ is a $q$-coloring of $K_{n,n}$ and the size of each color class of $c$ is either $a$ or $a+1$. Furthermore, we also have $\min\{|V_i\cap X|,|V_i\cap Y|\}\leq 1$ for every $1\leq i\leq q$. Hence $c$ is an equitable $(q,\infty,2)$-tree-coloring of $K_{n,n}$.
\end{proof}

\begin{lem}\label{lem:geq t+1}
The complete bipartite graph $K_{n,n}$ with $t(t+3)\leq 2n< (t+1)(t+4)$ has an equitable $(q,\infty,2)$-tree-coloring for every integer $q\geq 2\lfloor\frac{t+1}{2}\rfloor$.
\end{lem}

\begin{proof}
By Lemma \ref{lem:even}, we assume that $q$ is an odd integer. This implies $q\geq t+1$. If $2n=aq+r$, where $0\leq r\leq a-1$, then the two integers $a$ and $r$ would have the same parity. Note that $a=\frac{2n-r}{q}\leq \frac{2n}{q}<\frac{(t+1)(t+4)}{q}\leq t+4$ and $q\geq t+1$. We have
\begin{align}
r\leq a-2~{\rm and}~a\leq t+3. \label{ieq:1}
\end{align}
Now we prove the following two useful inequations:
\begin{align}
2q &\geq  a+r \label{ieq:2},\\
q+r &\geq a-1 \label{ieq:3}.
\end{align}
First,  if $a\leq t+2$, then $q+r\geq q\geq t+1\geq a-1$ and $2q\geq a+(a-2)\geq a+r$ by (\ref{ieq:1}). Similarly, if  $q\geq a-1$, then we would get the same results.
 Thus we assume that $a=t+3$ and $q\leq a-2$. Since $q\geq t+1=a-2$, $aq=(t+3)(t+1)$.  This implies that $r=2n-aq<(t+1)(t+4)-(t+1)(t+3)=t+1=a-2$, so $r\leq a-4$ and $2q=a+(a-4)\geq a+r$. On the other hand, $q$ and $a$ are both odd since $q=a-2$. It follows that $r=2n-aq>0$. Thus we have $q+r\geq q+1=a-1$.

The proof of this lemma is constructive. Let $X$ and $Y$ be two partite sets of $K_{n,n}$ as described in Lemma \ref{lem:eq.iff}. We are going to construct an equitable $(q,\infty, 2)$-tree-coloring of $K_{n,n}$ by distinguishing three cases.

Case 1. $q\leq 2r+1$.

We construct a coloring $c$ of $K_{n,n}$ by letting
\begin{align*}
|c(X_1)|=\frac{q-1}{2},~|c(Y_2)|=\frac{2q-a-r}{2},~|c(Y'_1)|=\frac{2r+1-q}{2},~|c(Y'_2)|=\frac{a-r}{2}
\end{align*}
and $|c(X_2)|=|c(X'_1)|=|c(X'_2)|=|c(Y_2)|=0$. Since $q\geq 1$, $2q\geq a+r$ by (\ref{ieq:2}), $2r+1\geq q$, $a-2\geq r$, $q$ is odd and $a, r$ have the same parity, the four values $|c(X_1)|,|c(Y_2)|,|c(Y'_1)|$ and $|c(Y'_2)|$ must be nonnegative integers. Moreover, one can easily check that the two equations (\ref{eq:2}) and (\ref{eq:3}) in Lemma \ref{lem:eq.iff} would hold by our choice. Thus $c$ is an equitable $(q,\infty, 2)$-tree-coloring of $K_{n,n}$.

Case 2. $2r+3\leq q\leq a+r-1$.

In this case we can construct a coloring $c$ of $K_{n,n}$ by letting
\begin{align*}
|c(X'_2)|=\frac{q+1}{2},~|c(Y_1)|=\frac{a+r-1-q}{2},~|c(Y_2)|=\frac{q-2r-1}{2},~|c(Y'_1)|=\frac{q+r-a+1}{2}
\end{align*}
and $|c(X_1)|=|c(X_2)|=|c(X'_1)|=|c(Y'_2)|=0$. One can easily see that $|c(X'_2)|, |c(Y_1)|$, $|c(Y_2)|$ and $|c(Y'_1)|$ are all nonnegative integers, since $2r+3\leq q\leq a+r-1$ and $q+r\geq a-1$ by (\ref{ieq:3}). On the other hand, the two equations (\ref{eq:2}) and (\ref{eq:3}) in Lemma \ref{lem:eq.iff} would also hold. Thus $c$ is an equitable $(q,\infty, 2)$-tree-coloring of $K_{n,n}$.

Case 3. $q\geq a+r+1$.

Now we construct a coloring $c$ of $K_{n,n}$ by setting
\begin{align*}
|c(X_2)|=\frac{q-1}{2},~|c(Y_2)|=\frac{q-a-r+1}{2},~|c(Y'_1)|=r,~|c(Y'_2)|=\frac{a-r}{2}
\end{align*}
and $|c(X_1)|=|c(X'_1)|=|c(X'_2)|=|c(Y_1)|=0$. One can easily check that $|c(X_2)|, |c(Y_2)|$, $|c(Y'_1)|$ and $|c(Y'_2)|$ are all nonnegative integers and the two equations (\ref{eq:2}) and (\ref{eq:3}) in Lemma \ref{lem:eq.iff} hold. Hence, $c$ is an equitable $(q,\infty, 2)$-tree-coloring of $K_{n,n}$.
\end{proof}

\begin{lem}\label{lem:t-impossible}
The complete bipartite graph $K_{n,n}$ with $2n=t(t+i)$, $i\geq 2$ and $t$ being odd has no equitable $(t,\infty, 2)$-tree-colorings.
\end{lem}

\begin{proof}
Suppose, to the contrary, that $K_{n,n}$ admits an equitable $(t,\infty, 2)$-tree-coloring $c$. Since $2n=t(t+i)$, the size of every color class of $c$ is exactly $t+i$. By Lemma \ref{lem:eq.iff}, without loss of generation, we can assume $|c(X_1)|+|c(X_2)|+|c(X'_1)|+|c(X'_2)|\geq \frac{t+1}{2}$. Here one should note that $t$ had been supposed to be odd. Thus we have $2n=2|X|\geq 2(t+i-1)(|c(X_1)|+|c(X_2)|+|c(X'_1)|+|c(X'_2)|)\geq (t+i-1)(t+1)=t(t+i)+i-1>t(t+i)=2n$, a contradiction.
\end{proof}

\begin{thm}\label{thm:main}
If $K_{n,n}$ is a complete bipartite graph and $k\geq 3$, then
$va^{\equiv}_{\infty,k }(K_{n,n})=va^{\equiv}_{\infty, 2}(K_{n,n})\leq 2\big\lfloor\frac{\lfloor\frac{-1+\sqrt{8n+9}}{2}\rfloor}{2}\big\rfloor$, and
furthermore, this bound is sharp.
\end{thm}

\begin{proof}
 Note that in any $(t,k,d)$-tree-coloring of $K_{n,n}$ the diameter of the subgraph induced by the vertices of any color class is at most 2, because otherwise we would find a 4-cycle in $K_{n,n}$ with its incident vertices receiving a same color, a contradiction. Therefore we have $va^{\equiv}_{k,\infty}(K_{n,n})=va^{\equiv}_{\infty, 2}(K_{n,n})$ for every $k\geq 2$.
Let $t= \lfloor\frac{-3+\sqrt{8n+9}}{2}\rfloor$. One can easily check that $t(t+3)\leq 2n<(t+1)(t+4)$. Hence by Lemma \ref{lem:geq t+1}, we have $va^{\equiv}_{\infty, 2}(K_{n,n})\leq 2\lfloor\frac{t+1}{2}\rfloor=2\big\lfloor\frac{\lfloor\frac{-1+\sqrt{8n+9}}{2}\rfloor}{2}\big\rfloor$. To show this bound is sharp, we investigate the graph $G:=K_{n,n}$ with $2n=t(t+3)$ and $t$ being odd. Note that  $2\big\lfloor\frac{\lfloor\frac{-1+\sqrt{8n+9}}{2}\rfloor}{2}\big\rfloor-1=t$, however, $G$ has no equitable $(t,\infty, 2)$-tree-colorings by Lemma \ref{lem:t-impossible}. This implies $va^{\equiv}_{\infty, 2}(G)=t+1$.
\end{proof}


\section{Planar graphs}

\begin{lem}\label{lem:label}
Let $S=\{v_1,\cdots,v_t\}$, where $v_1,\cdots,v_t$ are distinct vertices in $G$. If $G-S$ has an equitable $t$-tree-coloring and $|N_G(v_i)\setminus S|\leq 2i-1$ for every $1\leq i\leq t$, then $G$ has an equitable $t$-tree-coloring.
\end{lem}

\begin{proof}
Let $G_i=G\setminus \{v_1,\cdots,v_i\}$. It follows that $G=G_0$ and $G-S=G_t$. Let $c_t$ be an equitable $t$-tree-coloring of $G_t$. For every $t\geq i\geq 1$, we extend the equitable $t$-tree-coloring $c_i$ of $G_i$ to an equitable $t$-tree-coloring $c_{i-1}$ of $G_{i-1}$ by giving $v_i$ a color that is different from the colors in $\{c_i(v_{i+1}),\cdots,c_i(v_t)\}$ and that has been used on the neighbors of $v_i$ at most once. This is possible since $|N_G(v_i)\setminus S|\leq 2i-1$ for every $1\leq i\leq t$. After $t$ iterative extensions, one can check that the vertices in $S$ receive different colors under the final coloring $c_0$. Hence, $c_0$ is an equitable $t$-tree-coloring of $G$.
\end{proof}


\begin{lem}\label{mad103}
Every graph with maximum average degree less then $\frac{10}{3}$ contains at least one of the following configurations.\\
(C1.1) a vertex $x$ of degree $1$;\\
(C1.2) a $2$-vertex $x$ adjacent to a $6$-vertex $y$;\\
(C1.3) a $3$-vertex $x$ adjacent to a $4^-$-vertex $y$ and a $6^-$-vertex $z$;\\
(C1.4) an $i$-vertex $x$ adjacent to at least $i-1$ $2$-vertices, where $i=7,8,9$.
\end{lem}

\begin{proof}
Suppose, to the contrary, that $G$ contains none of the four configurations. It follows that $\delta(G)\geq 2$.
Assign initial charge $c(v)=d(v)$ to every vertex $v\in V(G)$. We now redistribute the charges of vertices in $G$ according to Rules 1 and 2 below. \\[.5em]
\noindent \textbf{Rule 1.} A $7^+$-vertex gives $\frac{2}{3}$ to each of its 2-neighbors.\\
\noindent \textbf{Rule 2.} A $4^+$-vertex gives $\frac{1}{6}$ to each of its 3-neighbors.\\[.5em]
Let $c'(v)$ be the charge of $v$ after discharging. Since (C1.2) is forbidden in $G$, every $2$-vertex is adjacent only to $7^+$-vertices in $G$. By Rule 1, we immediately have $c'(v)\geq 2+2\times \frac{2}{3}=\frac{10}{3}$ for every $2$-vertex $v$. Since the absence of (C1.3) in $G$ implies that every 3-vertex is adjacent to two $4^+$-vertices in $G$, $c'(v)\geq 3+2\times\frac{1}{6}=\frac{10}{3}$ for every $3$-vertex $v$ by Rule 2. Let $v$ be a vertex of degree between 4 and 6. By Rule 2, one can easily deduce that $c'(v)\geq d(v)-\frac{1}{6}d(v)\geq \frac{10}{3}$. Let $v$ be a vertex of degree between 7 and 9. Since (C1.4) is absent from $G$, $v$ is adjacent to at most $d(v)-2$ $2$-vertices, therefore, by Rules 1 and 2, we have $c'(v)\geq d(v)-\frac{2}{3}(d(v)-2)-2\times\frac{1}{6}\geq \frac{10}{3}$. At last, if $d(v)\geq 10$, then by Rules 1 and 2, $c'(v)\geq d(v)-\frac{2}{3}d(v)\geq \frac{10}{3}$. Hence, we have $\mad(G)\geq \frac{\sum_{v\in V(G)}c(v)}{|G|}=\frac{\sum_{v\in V(G)}c'(v)}{|G|}\geq \frac{10}{3}$, a contradiction.
\end{proof}

By Lemma \ref{mad103}, we have the following two immediate corollaries.

\begin{cor}\label{lem:girth5}
Every planar graph with girth at least $5$ contains at least one of four configurations mentioned in Lemma $\ref{mad103}$.
\end{cor}

\begin{cor}\label{cor:3-degerate}
Every planar graph with girth at least $5$ contains a vertex of degree at most $3$.
\end{cor}

\begin{thm}\label{thm:girth5}
If $G$ is a planar graph with girth at least 5, then $G$ has an equitable $t$-tree-coloring for every $t\geq 3$, that is, $va^{\equiv}_{\infty,\infty}(G)\leq 3$.
\end{thm}

\begin{proof}
By Corollary \ref{lem:girth5}, $G$ contains at least one of the configurations (C1.1)--(C1.4). In what follows, we prove the theorem by induction on the order of $G$, via assigning $t$ distinct vertices to $S=\{v_1,\cdots,v_t\}$ as described in Lemma \ref{lem:label}, where $t\geq 3$.

If $G$ contains the configuration (C1.1), then let $x:=v_1$. If $G$ contains the configuration (C1.2), then let $x:=v_1$ and $y:=v_t$. If $G$ contains the configuration (C1.3), then let $x:=v_1$, $y:=v_2$ and $w:=v_t$. If $G$ contains the configuration (C1.4) and $i=7$, then let $y:=v_1$, $z:=v_2$ and $x:=v_t$, where $y$ and $z$ are two 2-vertices that are adjacent to $x$. If $G$ contains the configuration (C4), $8\leq i\leq 9$ and $t\geq 4$, then let $y:=v_1$, $z:=v_2$ and $x:=v_t$, where $y$ and $z$ are two 2-vertices that are adjacent to $x$. Now in each case we fill the remaining unspecified positions in $S=\{v_1,v_2,\cdots,v_t\}$ from highest to lowest indices properly. Indeed, one can easily complete it by choosing at each step a vertex of degree at most 3 in the graph obtained from $G$ by deleting the vertices chosen for $S$ with higher indices. Corollary \ref{cor:3-degerate} guarantees that such vertices always exist. Meanwhile, by doing so, we would have $|N_G(v_i)\setminus \{v_{i+1},\cdots,v_t\}|\leq 2i-1$ for every $1\leq i\leq t$. Since $G-S$ is a planar graph with girth at least 5 and with order less that $G$, by induction hypothesis, $G-S$ has an equitable $t$-tree-coloring. Hence by Lemma \ref{lem:label}, $G$ also admits an equitable $t$-tree-coloring.

Now one should be care of that we have ignored two cases in the above discussions. There are the cases that $G$ contains configuration (C4), $8\leq i\leq 9$ and $t=3$. Let $x_1,\cdots,x_5$ be five 2-neighbors of $x$ in $G$. Consider the graph $G'=G-\{x,x_1,\cdots,x_5\}$. By induction, $G'$ has an equitable $3$-tree-coloring $c'$. If there is one color, say 1, which has not appeared on the vertex set $N_G(x)\setminus \{x_1,\cdots,x_5\}$ under the coloring $c'$, then we color $x, x_1$ by 1, $x_2, x_3$ by 2 and $x_4,x_5$ by 3. One can check that the extended coloring of $G$ is an equitable $3$-tree-coloring. Otherwise, since $|N_G(x)\setminus \{x_1,\cdots,x_5\}|=i-5\leq 4$, there are two colors, say $1$ and $2$, which have been used only once on the vertex set $N_G(x)\setminus \{x_1,\cdots,x_5\}$ under the coloring $c'$. Without loss of generality, denote the other neighbor of $x_1$ besides $x$ was colored by 1. We now color $x, x_1$ by 2, $x_2, x_3$ by 1 and $x_4,x_5$ by 3. One can also check that the resulting coloring of $G$ is an equitable $3$-tree-coloring.
\end{proof}

\begin{lem}\label{mad3}
Every graph with maximum average degree less then $3$ contains at least one of the following configurations.\\
(C2.1) a vertex $x$ of degree 1;\\
(C2.2) a $2$-vertex $x$ adjacent to a $4^-$-vertex $y$;\\
(C2.3) a 5-vertex $x$ adjacent to five 2-vertices .
\end{lem}

\begin{proof}
Suppose, to the contrary, that $G$ contains none of the four configurations. It follows that $\delta(G)\geq 2$.
Assign initial charge $c(v)=d(v)$ to every vertex $v\in V(G)$. We now redistribute the charges of vertices in $G$ according to the following rule.\\[.5em]
\noindent \textbf{Rule.} A $5^+$-vertex gives $\frac{1}{2}$ to each of its 2-neighbors.\\[.5em]
Let $c'(v)$ be the charge of $v$ after discharging. Since $G$ does not contain (C2.2), every 2-vertex is adjacent to two $5^+$-vertices in $G$. Therefore, $c'(v)\geq 2+2\times\frac{1}{2}=3$ for every 2-vertex $v$ by the discharging rule. Since 3-vertices and 4-vertices are not involved in the rule, $c'(v)=d(v)\geq 3$ for $3\leq d(v)\leq 4$. If $d(v)=5$, then $v$ is adjacent to at most four 2-vertices because of the absence of (C2.3) from $G$, so $c'(v)\geq d(v)-4\frac{1}{2}=3$. If $d(v)\geq 6$, then by the discharging rule, we still have $c'(v)\geq d(v)-\frac{1}{2}d(v)\geq 3$.
Hence, we have $\mad(G)\geq \frac{\sum_{v\in V(G)}c(v)}{|G|}=\frac{\sum_{v\in V(G)}c'(v)}{|G|}\geq 3$, a contradiction.
\end{proof}

By Lemma \ref{mad3}, we have the following immediate corollary.

\begin{cor}\label{lem:girth6}
Every planar graph with girth at least $6$ contains at least one of three configurations mentioned in Lemma $\ref{mad3}$.
\end{cor}

\begin{thm}\label{thm:girth6}
 If $G$ is a planar graph with girth at least 6, then $G$ has an equitable $t$-tree-coloring for every $t\geq 2$, that is, $va^{\equiv}_{\infty,\infty}(G)=2$ if $G$ is not a forest and $va^{\equiv}_{\infty,\infty}(G)=1$ otherwise.
\end{thm}

\begin{proof}
By Theorem \ref{thm:girth5}, we only need to show that $G$ has an equitable $2$-tree-coloring. We now apply induction on the order of $G$.

By Corollary \ref{lem:girth6}, $G$ contains one of the configurations among (C2.1), (C2.2) and (C2.3). If $G$ contains (C2.1), then by Corollary \ref{cor:3-degerate}, there exists a $3^-$-vertex $y$ in $G-x$. Now let $x:=v_1$ and $y:=v_2$. If $G$ contains (C2.2), then again let $x:=v_1$ and $y:=v_2$. In each case let $S=\{v_1,v_2\}$. We then have $|N_G(v_1)\setminus S|\leq 1$ and $|N_G(v_2)\setminus S|\leq 3$. Since $G-S$ has an equitable $2$-tree-coloring by induction, $G$ admits an equitable $2$-tree-coloring by Lemma \ref{lem:label}.

If $G$ contains (C2.3), then let $x_1,\cdots,x_5$ be the five 2-neighbors of $x$ and let $G'=G-\{x,x_1,x_2,x_3\}$. By induction, $G'$ has an equitable $2$-tree-coloring $c'$. If $c'(x_4)=c'(x_5)=1$, then color $x,x_1$ by 2 and $x_2,x_3$ by 1. If $c'(x_4)=1$ and $c'(x_5)=2$, then denote the other neighbor of $x_i$ besides $x$ be $x'_i$. If $c'(x'_1)=c'(x'_2)=c'(x'_3)=1$, then color $x,x_1$ by 2 and $x_2,x_3$ by 1. Otherwise, if $c'(x'_1)=1$ and $c'(x'_2)=c'(x'_3)=2$, then color $x,x_1$ by 2 and $x_2,x_3$ by 1. In each case one can check that the extended coloring of $G$ is an equitable $2$-tree-coloring.
\end{proof}

A graph is outerplanar if it can be drawn in the plane so that all vertices are lying on the outside face. It is easy to see that every outerplanar graph is planar. The following structural lemma for outerplanar graphs has been proved by many authors.

\begin{lem}\cite{BW}\label{lem:outer}
Every outerplanar graph with minimum degree at least two contains one of the following configurations:\\
(C1) two adjacent 2-vertices $u$ and $v$;\\
(C2) a 3-cycle $uvw$ with $d(u)=2$ and $d(v)=3$;\\
(C3) two intersecting 3-cycles $uvw$ and $xyw$ with $d(u)=d(x)=2$ and $d(w)=4$.
\end{lem}

From the above lemma, one can see that every outerplanar graph contains either a vertex $x$ of degree 1 or an edge $xy$ with $d(x)=2$ and $d(y)\leq 4$. Thus by a same argument as in Theorem \ref{thm:girth6}, we have the following theorem for outerplanar graphs.

\begin{thm}\label{thm:girth6}
Every outerplanar graph has an equitable $t$-tree-coloring for every $t\geq 2$, that is, $va^{\equiv}_{\infty,\infty}(G)=2$ if $G$ is not a forest and $va^{\equiv}_{\infty,\infty}(G)=1$ otherwise.
\end{thm}

\section{Concluding remarks and open problems}

First of all, we remark that the constructive proofs of Lemma \ref{lem:geq t+1}, Theorem \ref{thm:main1} and Theorem \ref{thm:main} yield linear-time algorithms to obtain an equitable $(2\big\lfloor\frac{\lfloor\frac{-1+\sqrt{8n+9}}{2}\rfloor}{2}\big\rfloor,k,\infty)$-tree-coloring for every $k\geq 2$ and an equitable $(2\lfloor\frac{n+1}{3}\rfloor,1,1)$-tree-coloring of $K_{n,n}$. Second, we would like to point out that the bounds for $va^{\equiv}_{1,1}(K_{n,n})$ and $va^{\equiv}_{\infty, 2}(K_{n,n})$ in Theorem \ref{thm:main1} and Theorem \ref{thm:main} are sharp in general case but they are not very tight for some special graphs. The examples will be shown after Theorem \ref{add2}.

Let $G=K_{n,n}$ be a complete bipartite graph and $a$ be an integer.  Now we consider the integral solution of the following equation on two nonnegative variables $x$ and $y$
\begin{align}
ax+(a+1)y=n.\label{fc}
\end{align}
For any solution $\mathbb{z}_{i}=(x_i,y_i)$ of the equation (\ref{fc}), define $z_i=x_i+y_i$. Then we have the following theorem.

\begin{thm}\label{add1}
Let $k,d\geq 1$ be two integers.
If the equation (\ref{fc}) has two integral solutions $\mathbb{z}_{1}$ and $\mathbb{z}_{2}$ (note needed to be different), then $K_{n,n}$ has an equitable $(z_1+z_2,k,d)$-tree-coloring with the size of each color class being either $a$ or $a+1$.
\end{thm}

 \begin{proof}
 Let $X$ and $Y$ be the two partite sets of $K_{n,n}$. Since $(x_1,y_1)$ and $(x_2,y_2)$ are two solutions of the equation (\ref{fc}), we can partition the vertices of $X$ into $X_1,\cdots X _{x_1},X_{x_1+1},\cdots,X_{z_1}$ and partition the vertices of $Y$ into $Y_1,\cdots Y _{x_2},Y_{x_2+1},\cdots,Y_{z_2}$ so that $|X_{i}|=|Y_{j}|=a$ and $|X_{k}|=|Y_{c}|=a+1$, where $1\leq i\leq x_1$, $1\leq j\leq x_2$, $x_1+1\leq k\leq z_1$ and $x_2+1\leq c\leq z_2$.
 One can easily see that such a partition implies an equitable $(z_1+z_2,k,d)$-tree-coloring of $K_{n,n}$.
\end{proof}

On the other hand, we can prove the following theorem.

\begin{thm}\label{add2}
Let $a\geq 3$ and $q$ be two integer. If $K_{n,n}$ has an equitable $(q,1,1)$-tree-coloring with the size of each color class being either $a$ or $a+1$, then the equation (\ref{fc}) has two integral solutions $\mathbb{z}_{1}$ and $\mathbb{z}_{2}$ (note needed to be different) so that $q=z_1+z_2$.
\end{thm}

\begin{proof}
Let $X$ and $Y$ be the two partite sets of $K_{n,n}$ and let $V_1,\cdots,V_q$ be the color classes of the given equitable $(q,1,1)$-tree-coloring $c$.
Since $a\geq 3$, for every $1\leq i\leq q$ we either have $V_i\subseteq X$ or $V_i\subseteq Y$. That is to say, every color class of $c$ is an independent set. This implies that $X$ or $Y$ can be partitioned into many parts so that the size of each part is either $a$ or $a+1$. Hence the equation (\ref{fc}) has two integral solutions $\mathbb{z}_{1}$ and $\mathbb{z}_{2}$ so that $q=z_1+z_2$.
\end{proof}

We consider the graph $K_{43,43}$ for example. Now set $a=3$ in the equation (\ref{fc}). We can collect all of the integral solutions of  the equation (\ref{fc}); they are (1,10), (5,7), (9,4), (13,1). By Theorem \ref{add1}, $K_{43,43}$ has an equitable $(t,1,1)$-tree-coloring for every $22\leq t\leq 28$. By Theorem \ref{thm:main1}, $K_{43,43}$ also has an equitable $(t,1,1)$-tree-coloring for every $t\geq 28$. On the other hand, by Theorem \ref{add2}, $K_{43,43}$ has no equitable $(21,1,1)$-tree-colorings. Thus $va^{\equiv}_{1,1}(K_{43,43})=22$.

Now we take the graph $K_{65,65}$ for another example. By Theorem \ref{thm:main}, $K_{65,65}$ has an equitable $(t,\infty,2)$-tree-coloring for every $t\geq 10$. In fact, we can also construct an equitable $(9,\infty,2)$-tree-coloring $c$ of $K_{65,65}$ by letting $|c(X'_2)|=|c(Y_1)|=5$ and $|c(X_1)|=|c(X_2)|=|c(X'_1)|=|c(Y_2)|=|c(Y'_1)|=|c(Y'_2)|=0$. Thus by Lemma \ref{lem:even}, $va^{\equiv}_{\infty, 2}(K_{65,65})\leq 8$. On the other hand, if $c$ is an equitable $(7,\infty,2)$-tree-coloring of $K_{65,65}$, then without loss of generality, we can assume that $|c(X_1)|+|c(X_2)|+|c(X'_1)|+|c(X'_2)|\geq 4$ by the equation (\ref{cup}), so we have $19|c(X_1)|+18|c(X_2)|+18|c(X'_1)|+17|c(X'_2)|+|Y'_1|+|Y'_2|\geq 17(|c(X_1)|+|c(X_2)|+|c(X'_1)|+|c(X'_2)|)\geq 68>n$. However, by the equation (\ref{eq:2}) of Lemma \ref{lem:eq.iff} we shall have $19|c(X_1)|+18|c(X_2)|+18|c(X'_1)|+17|c(X'_2)|+|Y'_1|+|Y'_2|=n$, which is a contradiction. Thus $K_{65,65}$ has no equitable $(7,\infty,2)$-tree-colorings and thus we have $va^{\equiv}_{\infty, 2}(K_{65,65})=8$.

Recall the results obtained in Section 3, we have proved that $va^{\equiv}_{\infty,\infty}(G)$ is bounded by a constant if $G$ is a planar graph with girth at least 5 or an outerplanar graph. Indeed, we believe this fact holds for every planar graph.

\begin{conj}\label{conj1}
$va^{\equiv}_{\infty,\infty}(G)=O(1)$ for every planar graph $G$.
\end{conj}

From the proof of Theorem \ref{thm:main1}, one can see that there exists complete bipartite graphs for which $va^{\equiv}_{\infty,\infty}(G)=\Omega(|G|^{\frac{1}{2}})$, so the above conjecture does not hold for general graphs. However, the following conjecture may be of interest.

\begin{conj}\label{conj2}
$va^{\equiv}_{\infty,\infty}(G)\leq \lceil\frac{\Delta(G)+1}{2}\rceil$ for every graph $G$.
\end{conj}

Since $va^{\equiv}_{\infty,\infty}(K_n)=\lceil\frac{n}{2}\rceil$ (this can be easily proved),
the upper bound in Conjecture \ref{conj2} is sharp if it holds.


\begin{thebibliography}{10}\setlength{\itemsep}{0pt}
\bibitem{Bondy.2008} J. A. Bondy and U. S. R. Murty, \textit{Graph Theory}, Springer, GTM 244, 2008.

\bibitem{BW} O. V. Borodin, D. R. Woodall, Thirteen coloring numbers of outerplane graphs, {\it Bull. Inst. Combin. Appl.} 14 (1995) 87--100.

\bibitem{Chartrand} G. Chartrand, H. V. Kronk, The point-arboricity of planar graphs. {\it J. London Math. Soc.}, 44, (1969), 612--616.

\bibitem{C-L-W} B. L. Chen, K. W. Lih, P. L. Wu, Equitable coloring and the maximum degree, {\it European J. Combin.}, 15 (1994) 443--447.

\bibitem{Chen}B. L. Chen, C. H. Yen, Equitable $\Delta$ coloring of graphs, {\it Discrete Mathematics}, 312 (2012) 1527--1517

\bibitem{Fan} H. Fan, H.A. Kierstead, G. Z. Liu, T. Molla, J. L. Wu, X. Zhang, A note on relaxed equitable coloring of graphs, Information Processing Letters 111(2011) 1062--1066

\bibitem{Hajnal} A. Hajnal, E. Szemerédi, Proof of a conjecture of P. Erd\H{o}s, \textit{In: Combinatorial Theory and its Applications} (P. Erd\H{o}s, A. Rényi and V. T. Sós, eds), North-Holand, London, 1970, pp. 601--623.

\bibitem{KiersteadJCT} H.A. Kierstead, A.V. Kostochka, An Ore-type theorem on equitable coloring, J. Combin. Theory Ser. B 98(2008) 226¨C234.
\bibitem{Kierstead} H. A. Kierstead, A. V. Kostochka, A short proof of the Hajnal-Szemerédi Theorem on equitable coloring, {\it Combinatorics, Probability and Computing}, 17(2008), 265--270.

\bibitem{KiersteadC} H. Kierstead, A. Kostochka, M. Mydlarz, E. Szemerédi, A fast algorithm for equitable coloring, Combinatorica 30(2010) 217--225.

\bibitem{Lih.bip} K.-W. Lih, P.-L. Wu, On equitable coloring of bipartite graphs, {\it Discrete Mathematics}, 151(1996), 155--160.


\bibitem{Meye} W. Meyer, Equitable coloring, {\it Amer. Math. Monthly}, 80, (1973), 920--922.

\bibitem{K} K. Nakprasit, Equitable colorings of planar graphs with maximum degree at least nine, {\it Discrete Mathematics}, 312 (2012) 1019--1024.

\bibitem{Wang} G. Wang, S. Zhou, J. L. Wu and G. Liu, Circular vertex arboricity,
{\it Discrete Applied Mathematics} 159 (2011) 1231--1238.

\bibitem{Yap-planar} H. P. Yap, Y. Zhang, Equitable colorings of planar graphs, {\it J. Combin. Math. Combin. Comput.}, 27, (1998), 97--105.

\bibitem{Yap} H. P. Yap, Y. Zhang, The equitable $\Delta$-coloring coniecture holds for outerplanar graphs, {\it Bull. Inst. Math. Acad. Sin.}, 25, (1997), 143--149.

\bibitem{ZW} X. Zhang, J.-L Wu, On equitable and equitable list colorings of series-parallel graphs, {\it Discrete Mathematics}, 311, (2011), 800--803

\bibitem{Yap} H. P. Yap, Y. Zhang, On equitable coloring of graphs, manuscript, 1996.

\bibitem{Zhu} J. L. Zhu and Y. H. Bu, Equitable list coloring of planar graphs without short cycles, {\it Theoretical Computer Science}, 407, (2008), 21-28.

\end{thebibliography}
\end{document}